\documentclass{aptpub}
\usepackage{amsmath,amsfonts}
\usepackage{amssymb}

\usepackage{graphicx,wrapfig,lipsum}
\usepackage[]{footmisc}

\usepackage{tikz}
\usetikzlibrary{automata, positioning}
\usetikzlibrary{arrows}
\usetikzlibrary{shapes.geometric, arrows}
\usetikzlibrary{positioning}
\usetikzlibrary{calc}
\usetikzlibrary{arrows}

\usepackage{scalerel}
\usepackage{multicol}
\usepackage{booktabs}
\usepackage{color}
\usepackage{dsfont}
\usepackage{multirow}

\makeatletter
\newcommand{\vast}{\bBigg@{3}}
\newcommand{\Vast}{\bBigg@{26}}
\makeatother
\usepackage{adjustbox}

\usepackage[colorlinks = true,linkcolor = blue,urlcolor  = blue,citecolor = blue,anchorcolor = blue]{hyperref}

\authornames{Mohammad Taha Shah} % insert the authors here for use in running head. If three or more authors please use (for example) M.~YARROW {\it et al.} Author names should follow the same M.~YARROW format and if two authors, separate by 'AND'.
\shorttitle{Counting state visits in Markov chains} % insert short title here for use in running head

\begin{document}%\recd{}{}%Do not alter this line.

\title{A Note on Exact State Visit Probabilities in Two-State Markov Chains} % insert title

\authorone[Indian Institute of Technology Delhi]{Mohammad Taha Shah} 
%Affiliation is just the name of your university or institution, for example 'University of Sheffield'. Author names should be of the form 'Mark Yarrow'. 
%Authors should be ordered alphabetically subject to the convention in that particular authors country. For example 'Remco van der Hofstad' would be listed under 'H' as is standard in the Netherlands. 

%Please use the following format for addresses and emails. The APT office will sort this out after you submit your files.
\addressone{Bharti School, IIT Delhi, Hauz Khas, New Delhi, India, 110016} % Your postal address goes here.
\emailone{tahashah@dbst.iitd.ac.in} %Authors email goes here.

\begin{abstract}
In this note we derive the exact probability that a specific state in a two-state Markov chain is visited exactly $k$ times after $N$ transitions. We provide a closed-form solution for $\mathbb{P}(N_l = k \mid N)$, considering initial state probabilities and transition dynamics. The solution corrects and extends prior incomplete results, offering a rigorous framework for enumerating state transitions. Numerical simulations validate the derived expressions, demonstrating their applicability in stochastic modeling.
\end{abstract}

%insert keywords separated by a semicolon. You should avoid including keywords which appear in the title.

% \ams{}{}    % insert the primary 2020 Maths Subject Classification number in the first bracket e.g. \ams{60E20}{49G03; 49F10}
                 % and the secondary ams number(s) in the second bracket
                 %Maximum of three in each, ideally one or two in each primary and secondary.
                 %codes found here ``https://mathscinet.ams.org/msnhtml/msc2020.pdf''

\section{Introduction}
Markov chains are fundamental tools in stochastic processes, providing a versatile framework for modeling systems that transition between a finite or countable number of states~\cite{norris1998markov}. This brief communication examines a two-state Markov chain, where the states are denoted as $S_l$ for $l \in \{0, 1\}$. Transitions between these states are governed by fixed probabilities, and the chain's behavior depends on these transition probabilities and the initial state distribution.

This work addresses the problem of calculating the probability that either state $S_0$ or $S_1$ is visited exactly $k$ times after $N$ transitions. The motivation for this investigation stems from the lack of a straightforward analytical result in the literature. During our exploration, we encountered a related discussion on StackExchange~\cite{mathStack}, which provided valuable insights and inspired our approach. However, the solution presented there contained errors and was incomplete. We have rectified these issues and present a comprehensive derivation supported by examples and numerical simulations\footnote{The codes for this note can be found in~\cite{markov_repo}}. This problem has significant applications in fields such as queuing theory, statistical physics, and bandit problems, where understanding state visit frequencies is crucial for analyzing system behavior.

The probability that state $S_l$ is visited exactly $k$ times after $N$ steps is denoted as $\mathbb{P}(N_l = k \mid N)$, where $l \in \{0, 1\}$. The Markov chain's initial state is determined by two possible states, $S_1$ and $S_0$, with corresponding probabilities $p_1$ and $p_0$. Consecutive transitions that keep the process in the same state are counted as additional visits to that state. The transition probability from state $S_i$ to state $S_j$ is denoted as $p_{ij}$, as shown in Fig.~\ref{fig:markov-chain}. Despite its apparent simplicity, determining the exact probability requires careful enumeration of state transitions.

\begin{figure}[t]
\centering
\begin{tikzpicture}[scale=1.2, transform shape, ->, >=stealth', auto, semithick, node distance=3cm]
\tikzstyle{every state}=[fill=none,draw=black,text=black]

\node[state] (A) {$S_0$};
\node[state] (B) [right of=A] {$S_1$};

\path (A) edge [loop left] node{$p_{00}$} (A)
          edge [bend left] node{$p_{01}$} (B)
      (B) edge [loop right] node{$p_{11}$} (B)
          edge [bend left] node{$p_{10}$} (A);          
\end{tikzpicture}
\caption{State diagram of a two-state Markov chain.}
\label{fig:markov-chain}
\end{figure}
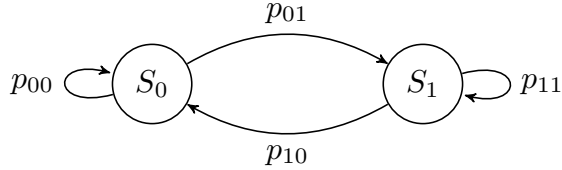

\section{Problem Formulation}
For our analysis, we derive the probability that state $S_1$ is visited exactly $k$ times after $N$ steps, i.e., $\mathbb{P}(N_1 = k \mid N)$. This analysis can be easily adapted for $S_0$. Quantifying visits to state $S_1$ involves enumerating state transitions originating from $S_1$. Specifically, it entails determining the frequency of transitions $S_1 \rightarrow S_1$ and $S_0 \rightarrow S_1$. Before considering transitions, the initial state of the Markov chain is either $S_0$ or $S_1$ with probabilities $p_0$ or $p_1$, respectively. This initial transition to state $S_0$ or $S_1$ is also considered a visit. For example, for $N=2$, if the initial state is $S_1$, then $\mathbb{P}(N_1 = 1 \mid N) = p_1 p_{10}$, and if the initial state is $S_0$, then $\mathbb{P}(N_1 = 1 \mid N) = p_0 p_{01}$. In the first case, state $S_1$ has already been visited once, while in the second case, the chain must transition from $S_0$ to $S_1$. 
% Additionally, revisiting state $S_1$ from $S_1$ counts as an additional visit. 
The following theorem formalizes the problem.

\begin{theorem}
\label{th:main_th}
Let $S_0$ and $S_1$ be the two states of a Markov chain, with transition probabilities $p_{00}$, $p_{01}$, $p_{10}$, and $p_{11}$, and initial state probabilities $p_0$ and $p_1$ for $S_0$ and $S_1$, respectively. The probability that state $S_1$ is visited exactly $k$ times after $N$ transitions is given by:
\begin{align}
    \mathbb{P} (N_1 = k \mid N) = p_1 \mathbb{P}_1 (k, N \mid S_1) + p_0 \mathbb{P}_2 (k, N \mid S_0)
    \label{eq:main}
\end{align}
where $\mathbb{P}_1 (k, N \mid S_1)$ and $\mathbb{P}_2 (k, N \mid S_0)$ are defined as follows:
\begin{align}
    \mathbb{P}_1 \left(k, N | S_1\right) = 
    \begin{cases}
        0 & k = 0 \\
        \sum_{j=1}^{c_1\left(k,N\right)} \binom{k-1}{j-1} p_{11}^{k-j} p_{10}^j \times \binom{N-k-1}{j-1} p_{01}^{j-1} p_{00}^{N-k-j} + \\ \hspace*{1cm} \sum_{j=1}^{c_2\left(k,N\right)} \binom{k-1}{j} p_{11}^{k-j-1} p_{10}^j \times \binom{N-k-1}{j-1} p_{01}^{j} p_{00}^{N-k-j} & 0 < k < N \\
        p_{11}^{N-1} & k=N
    \end{cases}
    \label{eq:eq_A}
\end{align}
\begin{align}
    \mathbb{P}_2 \left(k, N | S_0\right) = 
    \begin{cases}
        p_{00}^{N-1} & k = 0 \\
        \sum_{j=1}^{c_1\left(k,N\right)} \binom{k-1}{j-1} p_{11}^{k-j} p_{10}^{j-1} \times \binom{N-k-1}{j-1} p_{01}^{j} p_{00}^{N-k-j} + \\ \hspace*{1cm} \sum_{j=1}^{c_3\left(k,N\right)} \binom{k-1}{j-1} p_{11}^{k-j} p_{10}^j \times \binom{N-k-1}{j} p_{01}^{j} p_{00}^{N-k-j-1} & 0 < k < N \\
        0 & k=N
    \end{cases}
    \label{eq:eq_B}
\end{align}
Here, $c_1\left(k,N\right) = \min\left(k, N-k\right)$, $c_2\left(k,N\right) = \min\left(k-1, N-k\right)$, and 
\newline $c_3\left(k,N\right) = \min\left(k, N-k-1\right)$.
\end{theorem}

\begin{proof}
In~\eqref{eq:main}, the first term $p_1 \mathbb{P}_1 (k, N \mid S_1)$ is the product of the probability of choosing state $S_1$ initially and the probability that state $S_1$ is visited exactly $k$ times after $N-1$ steps, given that the initial state is $S_1$. Similarly, the second term $p_0 \mathbb{P}_2 (k, N \mid S_0)$ is the product of the probability of choosing state $S_0$ initially and the probability that state $S_1$ is visited exactly $k$ times after $N-1$ steps, given that the initial state is $S_0$. We break $\mathbb{P}_1 \left(k, N | S_1\right)$ and $\mathbb{P}_2 \left(k, N | S_0\right)$ into three cases. The first case, $k=0$, occurs when $S_1$ is visited $0$ times, which is only possible if all states visited in $N$ steps are $S_0$. Thus, $\mathbb{P}_1 \left(k, N | S_1\right) = 0$, and $\mathbb{P} (N_1 = k \mid N) = p_0 \mathbb{P}_2 (k, N \mid S_0) = p_0 p_{00}^{N-1}$. Similarly, the last case, $k = N$, occurs when all states visited in $N$ steps are $S_1$. Thus, $\mathbb{P}_2 (k, N \mid S_0) = 0$, and $\mathbb{P} (N_1 = k \mid N) = p_1 \mathbb{P}_1 (k, N \mid S_1) = p_1 p_{11}^{N-1}$. The remaining case, $0 < k < N$, requires careful analysis due to the many possible state transition combinations. We will discuss the upper limits of summation at the end of the proof. 

For $0 < k < N$,
\begin{align}
    \mathbb{P}_1 (k, N \mid S_1) = & \underbrace{\sum_{j=1}^{c_1\left(k,N\right)} \underbrace{\binom{k-1}{j-1} p_{11}^{k-j} p_{10}^j}_{A} \times \underbrace{\binom{N-k-1}{j-1} p_{01}^{j-1} p_{00}^{N-k-j}}_{B}}_{X} + \nonumber\\
    &\hspace*{1.5cm} \underbrace{\sum_{j=1}^{c_2\left(k,N\right)} \underbrace{\binom{k-1}{j} p_{11}^{k-j-1} p_{10}^j}_{C} \times \underbrace{\binom{N-k-1}{j-1} p_{01}^{j} p_{00}^{N-k-j}}_{D}}_{Y}
    \label{eq:eq_1}
\end{align}
and
\begin{align}
    \mathbb{P}_2 (k, N \mid S_0) = & \underbrace{\sum_{j=1}^{c_1\left(k,N\right)} \binom{k-1}{j-1} p_{11}^{k-j} p_{10}^{j-1} \times \binom{N-k-1}{j-1} p_{01}^{j} p_{00}^{N-k-j}}_{L} + \nonumber\\ 
    &\hspace*{1cm} \underbrace{\sum_{j=1}^{c_3\left(k,N\right)} \binom{k-1}{j-1} p_{11}^{k-j} p_{10}^j \times \binom{N-k-1}{j} p_{01}^{j} p_{00}^{N-k-j-1}}_{M}.
    \label{eq:eq_2}
\end{align}

Observing~\eqref{eq:eq_1} and~\eqref{eq:eq_2}, the expressions inside the summations $(X, L)$ and $(Y, M)$ are nearly identical, differing only in the summation limits, and few of the exponents of transition probabilities. Importantly, the first sums $(X, L)$ and the second sums $(Y, M)$ correspond to the final state being $S_0$ and $S_1$, respectively. We provide an explanation for~\eqref{eq:eq_1}, which can be similarly applied to $\mathbb{P}_2 (k, N \mid S_0)$. In~\eqref{eq:eq_1}, the expression inside the first summation $A \text{--} B$ corresponds to the final state being $S_0$. To have exactly $k$ visits to $S_1$, there must be precisely $k$, $S_0\rightarrow S_1$ and/or $S_1\rightarrow S_1$ transitions. Each time the chain is in $S_1$, it can either stay in $S_1$ ($S_1\rightarrow S_1$) or transition to $S_0$ ($S_1\rightarrow S_0$). Similarly, each time the chain is in $S_0$, it can either stay in $S_0$ ($S_0\rightarrow S_0$) or transition to $S_1$ ($S_0\rightarrow S_1$). Each transition from $S_0\rightarrow S_1$ creates a \textit{transition set} (defined later) where the chain can stay in $S_1$ for multiple steps before transitioning back to $S_0$.

In $X$, there are $j$ transitions $S_1\rightarrow S_0$, as the initial state is $S_1$ and the value of $k \neq 0$ or $N$ there has to be at least one $S_1\rightarrow S_0$ transitions. Given that the final state is $S_0$ and there are $j$ occurrences of the $S_1\rightarrow S_0$ transition, there must be $j-1$ transitions from $S_0 \rightarrow S_1$, as seen in the second summation $B$. The total transitions from $S_1$ to $S_1$ are $k-j$ i.e. given that there are $j$, $S_1\rightarrow S_0$ transitions and there have to be total of $k$ visits to state $S_1$ gives $k-j$, $S_1\rightarrow S_1$ transitions. Now, ff the total number of transitions is $N$, then the number of transitions from $S_0$ to $S_0$ must be $N$ minus the total transitions from $S_1\rightarrow S_1$, $S_1\rightarrow S_0$ and $S_1\rightarrow S_0$. Notice that in $X$, the sum of all exponents of transition probabilities is $N-1$ (i.e., $(k-1)+j+(j-1)+(N-k-j)$), and the same holds for $Y$ and $\mathbb{P}_2 \left(k, N | S_0\right)$. There are $N-1$ steps because the first step is choosing state $S_1$ or $S_0$.

A transition set is a sequence of consecutive transitions that occur within a specific state either $S_0$ or $S_1$ in a Markov chain. These transitions are grouped into sets based on the transitions between states $S_0\rightarrow S_1$ or $S_1\rightarrow S_0$ respectively. Now, focusing on $A$ and $B$, if there are $x$ transitions from $S_0\rightarrow S_1$, there will be $x+1$ transition sets of $S_1\rightarrow S_1$. The total number of $S_1\rightarrow S_1$ transition sets is $j$, as it is the number of $S_0\rightarrow S_1$ transitions plus one, i.e., $j-1+1$. The total number of $S_1\rightarrow S_1$ transitions is $k-j$, distributed across the $j$ transition sets created by $S_0\rightarrow S_1$ transitions. The distribution of $k-j$ transitions across $j$ transition sets is a problem of weak composition. A weak composition of a number $m$ into $n$ parts is a way of writing $m$ as the sum of $n$ non-negative integers. The number of weak compositions of $m$ into $n$ parts is given by $\binom{m+n-1}{n-1}$. In our case, $m=k-j$ and $n=j$, giving the number of weak compositions as $\binom{k-j+j-1}{j-1} = \binom{k-1}{j-1}$. Similarly, for part $B$, each $S_1\rightarrow S_0$ transition creates $j$ transition set where the chain can stay in $S_0$ for multiple steps. The number of ways to distribute $N-k-j$, $S_0\rightarrow S_0$ transitions across these sets is given by the weak composition $\binom{(N-k-j) + j - 1}{j-1} = \binom{N-k-1}{j-1}$. Since the distributions of $S_1\rightarrow S_1$ and $S_0\rightarrow S_0$ transitions are independent, the total number of valid transition sequences is the product of the two weak compositions, as seen in $X$.

Next, consider parts $C \text{--} D$ of~\eqref{eq:eq_1}, which follow a similar process, but the final state here is $S_1$ instead of $S_0$. The derivation of $Y$ follows a similar logic, with equal $j$ transitions from $S_1 \rightarrow S_0$ and $S_0 \rightarrow S_1$. Since the initial state and the final state is $S_1$, there are not $j-1$ transitions $S_0 \rightarrow S_1$ like in $A \text{--} B$ of~\eqref{eq:eq_1}. The total number of $S_1 \rightarrow S_1$ and $S_0 \rightarrow S_0$ transitions are $k-j-1$ and $N-k-j$, respectively. The total number of $S_1 \rightarrow S_1$ transition sets is $j+1$, and the total number of $S_0 \rightarrow S_0$ transition sets is $j$. Thus, the total number of possible combinations of having $k-j-1$, $S_1 \rightarrow S_1$ transitions spread over $j+1$ transition sets is the weak composition of natural numbers that sum to $k-j-1$, i.e., $\binom{(k-j-1)+(j+1)-1}{(j+1)-1} = \binom{k-1}{j}$. Similarly, the total number of possible combinations of having $N-k-j$, $S_0 \rightarrow S_0$ transitions spread over $j$ transition sets is given by $\binom{N-k-1}{j-1}$.

From~\eqref{eq:eq_1} and~\eqref{eq:eq_2}, three upper limits are used in the summations: $c_1(k,N)$, $c_2(k,N)$, and $c_3(k,N)$. Summations $X$ and $L$ share the same upper limit $c_1(k,N)$, while summations $Y$ and $M$ have upper limits $c_2(k,N)$ and $c_3(k,N)$, respectively. These limits combine two sub-cases: $k<\frac{N+1}{2}$ and $k \geq \frac{N+1}{2}$. When $k < \frac{N+1}{2}$, $c_1(k,N) = \min(k,N-k)=k$, $c_2(k,N) = \min(k-1,N-k)=k-1$, and $c_3(k,N) = \min(k,N-k-1)=k$. When $k \geq \frac{N+1}{2}$, $c_1(k,N) = \min(k,N-k)=N-k$, $c_2(k,N) = \min(k-1,N-k)=N-k$, and $c_3(k,N) = \min(k,N-k-1)=N-k-1$. This observation highlights that~\eqref{eq:eq_1} and~\eqref{eq:eq_2} differ in places like $j$ and $j-1$ or $N-k$ and $N-k-1$ depending on the value of $k$. This is because in~\eqref{eq:eq_1}, the initial state assumed is $S_1$, while in~\eqref{eq:eq_2}, the initial state assumed is $S_0$. This explains why there are $j$ and $j-1$ transitions from $S_1 \rightarrow S_0$ and $S_0 \rightarrow S_1$ in~\eqref{eq:eq_1}, respectively, while there are $j-1$ and $j$ transitions from $S_1 \rightarrow S_0$ and $S_0 \rightarrow S_1$ in~\eqref{eq:eq_2}, respectively. The same logic applies to $S_1 \rightarrow S_1$ and $S_0 \rightarrow S_0$ transitions, as seen in summations $Y$ and $M$. This completes the proof of Theorem~\ref{th:main_th} and explains all cases.
\end{proof}

\begin{figure}[t]
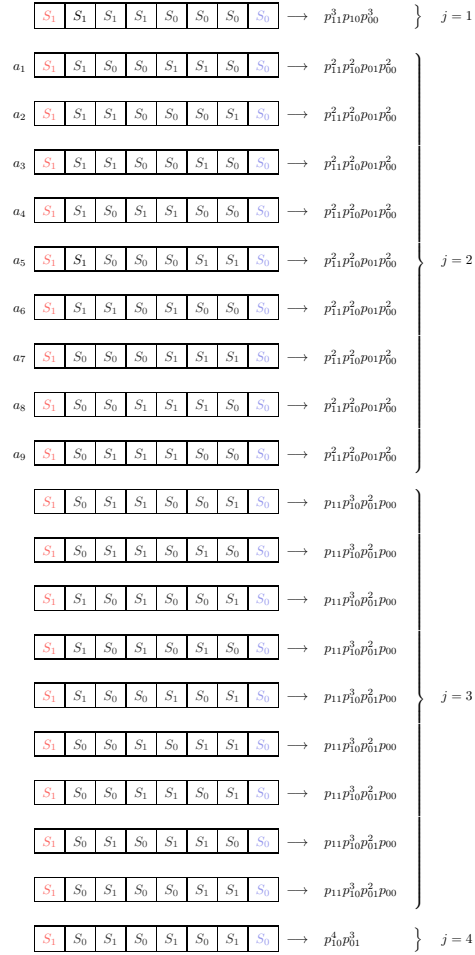

\centering
\scalebox{0.51}{
\centering
\begin{tabular}{lllllllllllll}
\cline{2-9}
\multicolumn{1}{l|}{}  & \multicolumn{1}{|l|}{{\color[HTML]{FD6864} $S_1$}} & \multicolumn{1}{l|}{$S_1$}                        & \multicolumn{1}{l|}{{\color[HTML]{333333} $S_1$}} & \multicolumn{1}{l|}{{\color[HTML]{333333} $S_1$}} & \multicolumn{1}{l|}{{\color[HTML]{333333} $S_0$}} & \multicolumn{1}{l|}{{\color[HTML]{333333} $S_0$}} & \multicolumn{1}{l|}{{\color[HTML]{333333} $S_0$}} & \multicolumn{1}{l|}{{\color[HTML]{9698ED} $S_0$}} & $\longrightarrow$ & $p_{11}^3 p_{10} p_{00}^3$ & $\Big\}$                    & $j=1$                    \\ \cline{2-9}
    &   &   &   &   &   &   &    &    &    &    &   \\ \cline{2-9}
\multicolumn{1}{l|}{$a_1$} & \multicolumn{1}{|l|}{{\color[HTML]{FD6864} $S_1$}} & \multicolumn{1}{l|}{{\color[HTML]{333333} $S_1$}} & \multicolumn{1}{l|}{{\color[HTML]{333333} $S_0$}} & \multicolumn{1}{l|}{{\color[HTML]{333333} $S_0$}} & \multicolumn{1}{l|}{{\color[HTML]{333333} $S_1$}} & \multicolumn{1}{l|}{{\color[HTML]{333333} $S_1$}} & \multicolumn{1}{l|}{{\color[HTML]{333333} $S_0$}} & \multicolumn{1}{l|}{{\color[HTML]{9698ED} $S_0$}} & $\longrightarrow$ & $p_{11}^2 p_{10}^2 p_{01} p_{00}^2$ &                       &                        \\ \cline{2-9}
    &   &   &   &   &   &   &    &    &    &    &   \\ \cline{2-9}
\multicolumn{1}{l|}{$a_2$} & \multicolumn{1}{|l|}{{\color[HTML]{FD6864} $S_1$}} & \multicolumn{1}{l|}{{\color[HTML]{333333} $S_1$}} & \multicolumn{1}{l|}{{\color[HTML]{333333} $S_1$}} & \multicolumn{1}{l|}{{\color[HTML]{333333} $S_0$}} & \multicolumn{1}{l|}{{\color[HTML]{333333} $S_0$}} & \multicolumn{1}{l|}{{\color[HTML]{333333} $S_0$}} & \multicolumn{1}{l|}{{\color[HTML]{333333} $S_1$}} & \multicolumn{1}{l|}{{\color[HTML]{9698ED} $S_0$}} & $\longrightarrow$ & $p_{11}^2 p_{10}^2 p_{01} p_{00}^2$ &                       &                        \\ \cline{2-9}
    &   &   &   &   &   &   &    &    &    &    &   \\ \cline{2-9}
\multicolumn{1}{l|}{$a_3$} & \multicolumn{1}{|l|}{{\color[HTML]{FD6864} $S_1$}} & \multicolumn{1}{l|}{{\color[HTML]{333333} $S_1$}} & \multicolumn{1}{l|}{{\color[HTML]{333333} $S_1$}} & \multicolumn{1}{l|}{{\color[HTML]{333333} $S_0$}} & \multicolumn{1}{l|}{{\color[HTML]{333333} $S_0$}} & \multicolumn{1}{l|}{{\color[HTML]{333333} $S_1$}} & \multicolumn{1}{l|}{{\color[HTML]{333333} $S_0$}} & \multicolumn{1}{l|}{{\color[HTML]{9698ED} $S_0$}} & $\longrightarrow$ & $p_{11}^2 p_{10}^2 p_{01} p_{00}^2$ &                       &                        \\ \cline{2-9}
    &   &   &   &   &   &   &    &    &    &    &   \\ \cline{2-9}
\multicolumn{1}{l|}{$a_4$} & \multicolumn{1}{|l|}{{\color[HTML]{FD6864} $S_1$}} & \multicolumn{1}{l|}{{\color[HTML]{333333} $S_1$}} & \multicolumn{1}{l|}{{\color[HTML]{333333} $S_0$}} & \multicolumn{1}{l|}{{\color[HTML]{333333} $S_1$}} & \multicolumn{1}{l|}{{\color[HTML]{333333} $S_1$}} & \multicolumn{1}{l|}{{\color[HTML]{333333} $S_0$}} & \multicolumn{1}{l|}{{\color[HTML]{333333} $S_0$}} & \multicolumn{1}{l|}{{\color[HTML]{9698ED} $S_0$}} & $\longrightarrow$ & $p_{11}^2 p_{10}^2 p_{01} p_{00}^2$ & &                        \\ \cline{2-9}
     &   &   &   &   &   &   &    &    &    &    &   \\ \cline{2-9}
\multicolumn{1}{l|}{$a_5$} & \multicolumn{1}{|l|}{{\color[HTML]{FD6864} $S_1$}} & \multicolumn{1}{l|}{$S_1$}                        & \multicolumn{1}{l|}{{\color[HTML]{333333} $S_0$}} & \multicolumn{1}{l|}{{\color[HTML]{333333} $S_0$}} & \multicolumn{1}{l|}{{\color[HTML]{333333} $S_0$}} & \multicolumn{1}{l|}{{\color[HTML]{333333} $S_1$}} & \multicolumn{1}{l|}{{\color[HTML]{333333} $S_1$}} & \multicolumn{1}{l|}{{\color[HTML]{9698ED} $S_0$}} & $\longrightarrow$ & $p_{11}^2 p_{10}^2 p_{01} p_{00}^2$ &                       &                        \\ \cline{2-9}
    &   &   &   &   &   &   &    &    &    &    &   \\ \cline{2-9}
\multicolumn{1}{l|}{$a_6$} & \multicolumn{1}{|l|}{{\color[HTML]{FD6864} $S_1$}} & \multicolumn{1}{l|}{{\color[HTML]{333333} $S_1$}} & \multicolumn{1}{l|}{{\color[HTML]{333333} $S_1$}} & \multicolumn{1}{l|}{{\color[HTML]{333333} $S_0$}} & \multicolumn{1}{l|}{{\color[HTML]{333333} $S_1$}} & \multicolumn{1}{l|}{{\color[HTML]{333333} $S_0$}} & \multicolumn{1}{l|}{{\color[HTML]{333333} $S_0$}} & \multicolumn{1}{l|}{{\color[HTML]{9698ED} $S_0$}} & $\longrightarrow$ & $p_{11}^2 p_{10}^2 p_{01} p_{00}^2$ &                       &                        \\ \cline{2-9}
    &   &   &   &   &   &   &    &    &    &    &   \\ \cline{2-9}
\multicolumn{1}{l|}{$a_7$} & \multicolumn{1}{|l|}{{\color[HTML]{FD6864} $S_1$}} & \multicolumn{1}{l|}{{\color[HTML]{333333} $S_0$}} & \multicolumn{1}{l|}{{\color[HTML]{333333} $S_0$}} & \multicolumn{1}{l|}{{\color[HTML]{333333} $S_0$}} & \multicolumn{1}{l|}{{\color[HTML]{333333} $S_1$}} & \multicolumn{1}{l|}{{\color[HTML]{333333} $S_1$}} & \multicolumn{1}{l|}{{\color[HTML]{333333} $S_1$}} & \multicolumn{1}{l|}{{\color[HTML]{9698ED} $S_0$}} & $\longrightarrow$ & $p_{11}^2 p_{10}^2 p_{01} p_{00}^2$ &                       &                        \\ \cline{2-9}
    &   &   &   &   &   &   &    &    &    &    &   \\ \cline{2-9}
\multicolumn{1}{l|}{$a_8$} & \multicolumn{1}{|l|}{{\color[HTML]{FD6864} $S_1$}} & \multicolumn{1}{l|}{{\color[HTML]{333333} $S_0$}} & \multicolumn{1}{l|}{{\color[HTML]{333333} $S_0$}} & \multicolumn{1}{l|}{{\color[HTML]{333333} $S_1$}} & \multicolumn{1}{l|}{{\color[HTML]{333333} $S_1$}} & \multicolumn{1}{l|}{{\color[HTML]{333333} $S_1$}} & \multicolumn{1}{l|}{{\color[HTML]{333333} $S_0$}} & \multicolumn{1}{l|}{{\color[HTML]{9698ED} $S_0$}} & $\longrightarrow$ & $p_{11}^2 p_{10}^2 p_{01} p_{00}^2$ &                       &                        \\ \cline{2-9}
    &   &   &   &   &   &   &    &    &    &    &   \\ \cline{2-9}
\multicolumn{1}{l|}{$a_9$} & \multicolumn{1}{|l|}{{\color[HTML]{FD6864} $S_1$}} & \multicolumn{1}{l|}{{\color[HTML]{333333} $S_0$}} & \multicolumn{1}{l|}{{\color[HTML]{333333} $S_1$}} & \multicolumn{1}{l|}{{\color[HTML]{333333} $S_1$}} & \multicolumn{1}{l|}{{\color[HTML]{333333} $S_1$}} & \multicolumn{1}{l|}{{\color[HTML]{333333} $S_0$}} & \multicolumn{1}{l|}{{\color[HTML]{333333} $S_0$}} & \multicolumn{1}{l|}{{\color[HTML]{9698ED} $S_0$}} & $\longrightarrow$ & $p_{11}^2 p_{10}^2 p_{01} p_{00}^2$ & \multirow{-17}{*}{\Vast\}} & \multirow{-17}{*}{$j=2$} \\ \cline{2-9}
    &   &   &   &   &   &   &    &    &    &    &   \\ \cline{2-9}
\multicolumn{1}{l|}{} & \multicolumn{1}{|l|}{{\color[HTML]{FD6864} $S_1$}} & \multicolumn{1}{l|}{{\color[HTML]{333333} $S_0$}} & \multicolumn{1}{l|}{{\color[HTML]{333333} $S_1$}} & \multicolumn{1}{l|}{{\color[HTML]{333333} $S_1$}} & \multicolumn{1}{l|}{{\color[HTML]{333333} $S_0$}} & \multicolumn{1}{l|}{{\color[HTML]{333333} $S_0$}} & \multicolumn{1}{l|}{{\color[HTML]{333333} $S_1$}} & \multicolumn{1}{l|}{{\color[HTML]{9698ED} $S_0$}} & $\longrightarrow$ & $p_{11} p_{10}^3 p_{01}^2 p_{00}$ &                       &                        \\ \cline{2-9}
    &   &   &   &   &   &   &    &    &    &    &   \\ \cline{2-9}
\multicolumn{1}{l|}{} & \multicolumn{1}{|l|}{{\color[HTML]{FD6864} $S_1$}} & \multicolumn{1}{l|}{{\color[HTML]{333333} $S_0$}} & \multicolumn{1}{l|}{{\color[HTML]{333333} $S_1$}} & \multicolumn{1}{l|}{{\color[HTML]{333333} $S_1$}} & \multicolumn{1}{l|}{{\color[HTML]{333333} $S_0$}} & \multicolumn{1}{l|}{{\color[HTML]{333333} $S_1$}} & \multicolumn{1}{l|}{{\color[HTML]{333333} $S_0$}} & \multicolumn{1}{l|}{{\color[HTML]{9698ED} $S_0$}} & $\longrightarrow$ & $p_{11} p_{10}^3 p_{01}^2 p_{00}$ &                       &                        \\ \cline{2-9}
    &   &   &   &   &   &   &    &    &    &    &   \\ \cline{2-9}
\multicolumn{1}{l|}{} & \multicolumn{1}{|l|}{{\color[HTML]{FD6864} $S_1$}} & \multicolumn{1}{l|}{{\color[HTML]{333333} $S_1$}} & \multicolumn{1}{l|}{{\color[HTML]{333333} $S_0$}} & \multicolumn{1}{l|}{{\color[HTML]{333333} $S_1$}} & \multicolumn{1}{l|}{{\color[HTML]{333333} $S_0$}} & \multicolumn{1}{l|}{{\color[HTML]{333333} $S_0$}} & \multicolumn{1}{l|}{{\color[HTML]{333333} $S_1$}} & \multicolumn{1}{l|}{{\color[HTML]{9698ED} $S_0$}} & $\longrightarrow$ & $p_{11} p_{10}^3 p_{01}^2 p_{00}$ &                       &                        \\ \cline{2-9}
    &   &   &   &   &   &   &    &    &    &    &   \\ \cline{2-9}
\multicolumn{1}{l|}{} & \multicolumn{1}{|l|}{{\color[HTML]{FD6864} $S_1$}} & \multicolumn{1}{l|}{{\color[HTML]{333333} $S_1$}} & \multicolumn{1}{l|}{{\color[HTML]{333333} $S_0$}} & \multicolumn{1}{l|}{{\color[HTML]{333333} $S_1$}} & \multicolumn{1}{l|}{{\color[HTML]{333333} $S_0$}} & \multicolumn{1}{l|}{{\color[HTML]{333333} $S_1$}} & \multicolumn{1}{l|}{{\color[HTML]{333333} $S_0$}} & \multicolumn{1}{l|}{{\color[HTML]{9698ED} $S_0$}} & $\longrightarrow$ & $p_{11} p_{10}^3 p_{01}^2 p_{00}$ &                       &                        \\ \cline{2-9}
    &   &   &   &   &   &   &    &    &    &    &   \\ \cline{2-9}
\multicolumn{1}{l|}{} & \multicolumn{1}{|l|}{{\color[HTML]{FD6864} $S_1$}} & \multicolumn{1}{l|}{{\color[HTML]{333333} $S_1$}} & \multicolumn{1}{l|}{{\color[HTML]{333333} $S_0$}} & \multicolumn{1}{l|}{{\color[HTML]{333333} $S_0$}} & \multicolumn{1}{l|}{{\color[HTML]{333333} $S_1$}} & \multicolumn{1}{l|}{{\color[HTML]{333333} $S_0$}} & \multicolumn{1}{l|}{{\color[HTML]{333333} $S_1$}} & \multicolumn{1}{l|}{{\color[HTML]{9698ED} $S_0$}} & $\longrightarrow$ & $p_{11} p_{10}^3 p_{01}^2 p_{00}$ &                       &                        \\ \cline{2-9}
    &   &   &   &   &   &   &    &    &    &    &   \\ \cline{2-9}
\multicolumn{1}{l|}{} & \multicolumn{1}{|l|}{{\color[HTML]{FD6864} $S_1$}} & \multicolumn{1}{l|}{{\color[HTML]{333333} $S_0$}} & \multicolumn{1}{l|}{{\color[HTML]{333333} $S_0$}} & \multicolumn{1}{l|}{{\color[HTML]{333333} $S_1$}} & \multicolumn{1}{l|}{{\color[HTML]{333333} $S_0$}} & \multicolumn{1}{l|}{{\color[HTML]{333333} $S_1$}} & \multicolumn{1}{l|}{{\color[HTML]{333333} $S_1$}} & \multicolumn{1}{l|}{{\color[HTML]{9698ED} $S_0$}} & $\longrightarrow$ & $p_{11} p_{10}^3 p_{01}^2 p_{00}$ &                       &                        \\ \cline{2-9}
    &   &   &   &   &   &   &    &    &    &    &   \\ \cline{2-9}
\multicolumn{1}{l|}{} & \multicolumn{1}{|l|}{{\color[HTML]{FD6864} $S_1$}} & \multicolumn{1}{l|}{{\color[HTML]{333333} $S_0$}} & \multicolumn{1}{l|}{{\color[HTML]{333333} $S_0$}} & \multicolumn{1}{l|}{{\color[HTML]{333333} $S_1$}} & \multicolumn{1}{l|}{{\color[HTML]{333333} $S_1$}} & \multicolumn{1}{l|}{{\color[HTML]{333333} $S_0$}} & \multicolumn{1}{l|}{{\color[HTML]{333333} $S_1$}} & \multicolumn{1}{l|}{{\color[HTML]{9698ED} $S_0$}} & $\longrightarrow$ & $p_{11} p_{10}^3 p_{01}^2 p_{00}$ &                       &                        \\ \cline{2-9}
    &   &   &   &   &   &   &    &    &    &    &   \\ \cline{2-9}
\multicolumn{1}{l|}{} & \multicolumn{1}{|l|}{{\color[HTML]{FD6864} $S_1$}} & \multicolumn{1}{l|}{{\color[HTML]{333333} $S_0$}} & \multicolumn{1}{l|}{{\color[HTML]{333333} $S_1$}} & \multicolumn{1}{l|}{{\color[HTML]{333333} $S_0$}} & \multicolumn{1}{l|}{{\color[HTML]{333333} $S_1$}} & \multicolumn{1}{l|}{{\color[HTML]{333333} $S_1$}} & \multicolumn{1}{l|}{{\color[HTML]{333333} $S_0$}} & \multicolumn{1}{l|}{{\color[HTML]{9698ED} $S_0$}} & $\longrightarrow$ & $p_{11} p_{10}^3 p_{01}^2 p_{00}$ &                       &                        \\ \cline{2-9}
    &   &   &   &   &   &   &    &    &    &    &   \\ \cline{2-9}
\multicolumn{1}{l|}{} & \multicolumn{1}{|l|}{{\color[HTML]{FD6864} $S_1$}} & \multicolumn{1}{l|}{{\color[HTML]{333333} $S_0$}} & \multicolumn{1}{l|}{{\color[HTML]{333333} $S_1$}} & \multicolumn{1}{l|}{{\color[HTML]{333333} $S_0$}} & \multicolumn{1}{l|}{{\color[HTML]{333333} $S_0$}} & \multicolumn{1}{l|}{{\color[HTML]{333333} $S_1$}} & \multicolumn{1}{l|}{{\color[HTML]{333333} $S_1$}} & \multicolumn{1}{l|}{{\color[HTML]{9698ED} $S_0$}} & $\longrightarrow$ & $p_{11} p_{10}^3 p_{01}^2 p_{00}$ & \multirow{-17}{*}{\Vast\}} & \multirow{-17}{*}{$j=3$} \\ \cline{2-9}
    &   &   &   &   &   &   &    &    &    &    &   \\ \cline{2-9}
\multicolumn{1}{l|}{} & \multicolumn{1}{|l|}{{\color[HTML]{FD6864} $S_1$}} & \multicolumn{1}{l|}{{\color[HTML]{333333} $S_0$}} & \multicolumn{1}{l|}{{\color[HTML]{333333} $S_1$}} & \multicolumn{1}{l|}{{\color[HTML]{333333} $S_0$}} & \multicolumn{1}{l|}{{\color[HTML]{333333} $S_1$}} & \multicolumn{1}{l|}{{\color[HTML]{333333} $S_0$}} & \multicolumn{1}{l|}{{\color[HTML]{333333} $S_1$}} & \multicolumn{1}{l|}{{\color[HTML]{9698ED} $S_0$}} & $\longrightarrow$ & $p_{10}^4 p_{01}^3$ & \Big\}                    & $j=4$                    \\ \cline{2-9}
\end{tabular}}
\caption{All the possible state transition paths considering the initial state is $S_1$ and final state is $S_0$.}
\label{fig:fig_2}
\end{figure}

\section{Example}

To better understand the transition sets and possible combinations, consider the example of $k=4$ and $N = 8$. We enumerate all possible state transition paths for part $X$ of~\eqref{eq:eq_1}. Since $4 < \frac{8+1}{2} = 4.5$, the $X$ part of~\eqref{eq:eq_1} is:
\begin{align}
    \sum_{j=1}^{4} \binom{3}{j-1} p_{11}^{4-j} p_{10}^j \times \binom{3}{j-1} p_{01}^{j-1} p_{00}^{4-j} &= \underbrace{p_{11}^3 p_{10} \times p_{00}^3}_{j=1} + \underbrace{3p_{11}^2 p_{10}^2 \times 3p_{01}p_{11}^2}_{j=2} + \nonumber\\
    &\hspace*{1.5cm} \underbrace{3p_{11} p_{10}^3 \times 3p_{01}^2p_{11}}_{j=3} + \underbrace{p_{10}^4 \times p_{01}}_{j=4}
    \label{eq:eq_3}
\end{align}
Fig.~\ref{fig:fig_2} shows all possible state transition paths if the initial state is $S_1$ and the final state is $S_0$. It has already been seen that the final state in the $X$ part of the summation is $S_0$, so Fig.~\ref{fig:fig_2} illustrates all possible combinations for this scenario. For $j=1$, there are no $S_0 \rightarrow S_1$ transitions, so there is a single $S_1 \rightarrow S_1$ transition set spread over three $S_1 \rightarrow S_1$ transitions, giving the weak composition as $1$. Similarly, the logic applies to the weak composition of $S_0 \rightarrow S_0$ transitions. For $j=2$, there are nine combinations, as seen in Fig.~\ref{fig:fig_2} and verified by~\eqref{eq:eq_3}. Two $S_1 \rightarrow S_1$ transition sets are spread over two $S_1 \rightarrow S_1$ transitions, giving the weak composition $3$. The weak compositions of $2$ into $2$ parts are
\begin{enumerate}
    \item $(2,0)$: $2$ transitions in the first set, $0$ in the second.
    \item $(1,1)$: $1$ transition in the first set, $1$ in the second.
    \item $(0,2)$: $0$ transitions in the first set, $2$ in the second.
\end{enumerate}
From Fig.~\ref{fig:fig_2}, transition paths $a_2,a_3,a_6$ correspond to $(2,0)$, paths $a_1,a_4,a_5$ correspond to $(1,1)$, and paths $a_7,a_8,a_9$ correspond to $(0,2)$. Similarly, two $S_0 \rightarrow S_0$ transition sets are spread over two $S_0 \rightarrow S_0$ transitions, giving the weak composition as $3$. The placement of $(2,0)$, $(1,1)$, and $(0,2)$ $S_1 \rightarrow S_1$ transition sets, given there are two $S_0 \rightarrow S_0$ transition sets results in $3 \times 3 = 9$ total possible combinations. The same logic follows for $j=3$ and $j=4$.

\section{Conclusion}
In this note, we derived the exact probability that a specific state in a two-state Markov chain is visited exactly $k$ times after $N$ transitions. Our closed-form solution for $\mathbb{P} (N_l = k \mid N)$ provides a rigorous and comprehensive framework to analyze state visit frequencies, correcting previous incomplete results found in the literature. Through careful enumeration of state transitions and numerical validation, we demonstrated the applicability of our results to stochastic modeling. This work has potential implications in various domains, such as queuing systems, statistical physics, and reinforcement learning, where understanding state visit frequencies is essential. Future work can extend these results to multi-state Markov chains or investigate applications in broader stochastic processes. Furthermore, the derived probabilities may serve as a foundation for optimizing decision-making strategies in Markovian environments.

\bibliographystyle{APT}
\bibliography{references2}

\end{document}